\documentclass[11pt]{amsart}
\usepackage{amsmath}%
\usepackage{amsfonts}%
\usepackage{amssymb}%
\usepackage{graphicx, color, enumerate}
\usepackage{latexsym}
\usepackage{amscd}
\usepackage{mathrsfs}
\usepackage{cancel}
\usepackage{mathtools}
\usepackage{xcolor}
\usepackage[top=1in, bottom=1.25in, left=1.25in, right=1.25in]{geometry}
\usepackage[pdfborder=000,pdftex=true]{hyperref}
\newtheorem{theorem}{Theorem}[section]
\theoremstyle{plain}
\newtheorem*{theorem*}{Theorem}
\theoremstyle{plain}

\newtheorem{definition}{Definition}[section]
\newtheorem{lemma}{Lemma}[section]

\newtheorem{remark}{Remark}[section]

\numberwithin{equation}{section}
\DeclareRobustCommand{\rchi}{{\mathpalette\irchi\relax}}
\newcommand{\irchi}[2]{\raisebox{\depth}{$#1\chi$}}

\title[A Strong Maximum Principle for the Fractional Laplacian]{A Strong Maximum Principle for the fractional Laplace equation with Mixed Boundary Condition}

\author{Rafael López-Soriano}
\author{Alejandro Ortega}
\email[R. López-Soriano]{ralopezs@math.uc3m.es}%
\email[A. Ortega ]{alortega@math.uc3m.es}
\address[R. López-Soriano, A. Ortega]{Departamento de Matem\'aticas,  
Universidad Carlos III de Madrid, Av. Universidad 30, 28911 Legan\'es (Madrid), Spain}

\thanks{This work has been supported by the Madrid Government (Comunidad de Madrid-Spain) under the Multiannual Agreement with UC3M in the line of Excellence of University Professors (EPUC3M23), and in the context of the V PRICIT (Regional Programme of Research and Technological Innovation).\\ The authors are partially supported by the Ministry of Economy and Competitiveness of Spain, under research project PID2019-106122GB-I00. }

\begin{document}
\maketitle

\begin{abstract}
In this work we prove a strong maximum principle for fractional elliptic problems with mixed Dirichlet--Neumann boundary data which extends the one proved by J. D\'avila (cf. \cite{Davila2001}) to the fractional setting. In particular, we present a comparison result for two solutions of the fractional Laplace equation involving the spectral fractional Laplacian endowed with homogeneous mixed boundary condition. This result represents a non--local counterpart to a Hopf's Lemma for fractional elliptic problems with mixed boundary data.
\end{abstract}

\section{Introduction}
The aim of this paper is to prove a strong maximum principle for elliptic problems involving a fractional Laplacian operator and homogeneous mixed boundary data. In particular, we consider the problem
\begin{equation}\label{problema1}\tag{$P_s$}
        \left\{
        \begin{tabular}{rcl}
        $(-\Delta)^su=f$ & &\mbox{in } $\Omega$, \\
        $B(u)=0$  & &on $\partial\Omega$,
        \end{tabular}
        \right.
\end{equation}
where  $\frac{1}{2}<s<1$, $f\in \mathcal{C}_0^{\infty}(\Omega)$, $f\gneq 0$ and $\Omega$ is a smooth bounded domain of $\mathbb{R}^N$ with $N\geq2$. Here $(-\Delta)^s$ denotes the \textit{spectral fractional Laplacian} defined through the spectral decomposition of the classical Laplacian with mixed Dirichlet--Neumann boundary condition $B(u)$ (see Section \ref{Functionalsetting} for further details) given by
\begin{equation*}
B(u)=u\rchi_{\Sigma_{\mathcal{D}}}+\frac{\partial u}{\partial \nu} \rchi_{\Sigma_{\mathcal{N}}},
\end{equation*}
where $\nu$ is the outward normal to $\partial\Omega$ and $\chi_A$ stands for the characteristic function of the set $A\subset\partial\Omega$. The sets $\Sigma_{\mathcal{D}}$ and $\Sigma_{\mathcal{N}}$ satisfy the following
\begin{itemize}
\item $\Sigma_{\mathcal{D}}$ and $\Sigma_{\mathcal{N}}$ are $(N-1)$-dimensional smooth submanifolds of $\partial\Omega$,
\item $\Sigma_{\mathcal{D}}$ is a closed (with respect to the relative topology) manifold of positive $(N-1)$-dimensional Lebesgue measure,
\item $|\Sigma_{\mathcal{D}}|=\alpha\in(0,|\partial\Omega|)$,
\item $\Sigma_{\mathcal{D}}\cap\Sigma_{\mathcal{N}}=\emptyset$, $\Sigma_{\mathcal{D}}\cup\Sigma_{\mathcal{N}}=\partial\Omega$ and $\Sigma_{\mathcal{D}}\cap\overline{\Sigma}_{\mathcal{N}}=\Gamma$,
\item $\Gamma$ a smooth $(N-2)$-dimensional submanifold of $\partial\Omega$.
\end{itemize}
As in the local case, i.e. $s=1$, by comparison one can easily prove that, for $C_1=\max\limits_{x\in\Omega}f(x)$,
\begin{equation*}
u(x)\leq C_1 v(x)\quad\text{for } x\in\Omega,
\end{equation*}
with $v$ being the solution to 
\begin{equation}\label{problema_v_1}     
\left\{
        \begin{tabular}{rcl}
        $(-\Delta)^sv=1$ & &\mbox{in } $\Omega$, \\
        $B(v)=0$  & &on $\partial\Omega$.
        \end{tabular}
        \right.
\end{equation}
So a natural question is whether the opposite inequality, namely,
\begin{equation}\label{opositeineq}
v(x)\leq C_2 u(x)\quad\text{for } x\in\Omega,
\end{equation}
holds true for some constant $C_2>0$.

For the local case, D\'avila (cf. \cite{Davila2001}) proved that inequality \eqref{opositeineq} holds for a positive constant $C_2$ depending on $\Omega$, $\Sigma_{\mathcal{D}}$, $\Sigma_{\mathcal{N}}$ and $\|fv\|_{L^1(\Omega)}$. We will obtain a similar result for the mixed boundary data problem \eqref{problema1} by adapting the approach of D\'avila to our fractional setting. To that end we will also use the regularity results proved in \cite{Carmona2020}.\newline
Let us remark that in \cite{Barrios2020} the authors proved a fractional strong maximum principle, but dealing with a different fractional operator which is defined by means of a singular integral.\newline
Our main aim is then to prove the following.
\begin{theorem}\label{teo:dav}
Assume that $f\in C_0^\infty(\Omega)$, $f\geq0$ and let  $u$ be the solution to 
\eqref{problema1}. Then there exists a constant $C=C(N,s,\Omega,\Sigma_{\mathcal{D}},\Sigma_{\mathcal{N}})>0$ such that
\begin{equation*}
u(x)\geq C \left(\int_\Omega f v dz\right)v(x) \qquad \mbox{for } x\in\Omega\,,
\end{equation*}
being $v$ the solution to problem \eqref{problema_v_1}.
\end{theorem}

The key result we need to prove to obtain Theorem \ref{teo:dav} is an $L^{\infty}$ bound on the ratio between the solution to \eqref{problema1} with a nonnegative $f\in L^{\infty} (\Omega)$ and the solution to a suitable auxiliary problem. In particular, let $v\in H_{\Sigma_{\mathcal{D}}}^s(\Omega)$ be the solution to 
\begin{equation}\label{prob_v_com_1}
\left\{
        \begin{tabular}{rcl}
        $(-\Delta)^sv=g$ & &\mbox{in } $\Omega$, \\
        $B(v)=0$  & &on $\partial\Omega$,
        \end{tabular}
        \right.
\end{equation}
with $g\in L^p(\Omega)$, $p>N/s$ and $g\gneq 0$. Then, the next result holds. 

\begin{theorem}\label{lem:davila2}
Let $u$ be the solution to \eqref{problema1} with $f\in L^\infty(\Omega)$, $f\gneq0$ and $g\in L^p(\Omega)$ for some $p>N/s$ and let $v$ be the solution to \eqref{prob_v_com_1}. Then there exists a constant $C>0$ such that 
\begin{equation*}
\left\| \ \frac vu   \  \right \|_ {L^\infty(\Omega)}\leq C\|g\|_{L^p(\Omega)}
\end{equation*}
with $C$ depending on $N$, $p$, $s$, $\Omega$, $\Sigma_{\mathcal{D}}$, $\|u\|_{L^{\infty}(\Omega)}$, $\|f\|_{L^{\infty}(\Omega)}$ and $1/(\int_{\Omega}f(z)d(z)dz)$ where $d(x)=\textrm{dist}(x,\partial\Omega)$.
\end{theorem}


\section{Functional setting and preliminaries}\label{Functionalsetting}

As far as the fractional Laplace operator is concerned, we recall its definition given through the spectral decomposition. Let
$(\varphi_i,\lambda_i)$ be the eigenfunctions (normalized with respect to the $L^2(\Omega)$-norm) and the eigenvalues of
$(-\Delta)$ equipped with homogeneous mixed Dirichlet--Neumann boundary data, respectively. Then, $(\varphi_i,\lambda_i^s)$ are the eigenfunctions and eigenvalues of the fractional operator
$(-\Delta)^s$, where, given $\displaystyle u_i(x)=\sum_{j\geq1}\langle u_i,\varphi_j\rangle\varphi_j$, $i=1,2$, it holds
\begin{equation*}
\langle(-\Delta)^s u_1, u_2\rangle=\sum_{j\ge 1} \lambda_j^s\langle u_1,\varphi_j\rangle \langle u_2,\varphi_j\rangle,
\end{equation*}
i.e., the action of the fractional operator on a smooth function $u_1$ is given by
\begin{equation*}
(-\Delta)^su_1=\sum_{j\ge 1} \lambda_j^s\langle u_1,\varphi_j\rangle\varphi_j.
\end{equation*}
As a consequence, the fractional Laplace operator $(-\Delta)^s$ is well defined through its spectral decomposition in the
following space of functions that vanish on $\Sigma_{\mathcal{D}}$,
\begin{equation*}
H_{\Sigma_{\mathcal{D}}}^s(\Omega)=\left\{u=\sum_{j\ge 1} a_j\varphi_j\in L^2(\Omega):\ \|u\|_{H_{\Sigma_{\mathcal{D}}^s}(\Omega)}^2=
\sum_{j\ge 1} a_j^2\lambda_j^s<\infty\right\}.
\end{equation*}
Observe that since $u\in H_{\Sigma_{\mathcal{D}}}^s(\Omega)$, it follows that
\begin{equation*}
\|u\|_{H_{\Sigma_{\mathcal{D}}}^s(\Omega)}=\left\|(-\Delta)^{\frac{s}{2}}u\right\|_{L^2(\Omega)}.
\end{equation*}
As it is proved in \cite[Theorem 11.1]{Lions1972}, if $0<s\le \frac{1}{2}$ then $H_0^s(\Omega)=H^s(\Omega)$ and, therefore, also
$H_{\Sigma_{\mathcal{D}}}^s(\Omega)=H^s(\Omega)$, while for $\frac 12<s<1$, $H_0^s(\Omega)\subsetneq H^s(\Omega)$. Hence,
the range $\frac 12<s<1$ guarantees that $H_{\Sigma_{\mathcal{D}}}^s(\Omega)\subsetneq H^s(\Omega)$ and it provides us the correct functional space to study the mixed boundary problem \eqref{problema1}.\newline
This definition of the fractional powers of the Laplace operator allows us to integrate by parts in the appropriate spaces, so that a natural definition of weak solution to problem \eqref{problema1} is the following.
\begin{definition}
We say that $u\in H_{\Sigma_{\mathcal{D}}}^s(\Omega)$ is a solution to \eqref{problema1} if
\begin{equation*}
\int_{\Omega}(-\Delta)^{s/2}u \,(-\Delta)^{s/2}\psi dx=\int_{\Omega}f\psi dx\quad \text{for any}\ \psi\in H_{\Sigma_{\mathcal{D}}}^s(\Omega).
\end{equation*}
\end{definition}
Due to the nonlocal nature of the fractional operator $(-\Delta)^s$ some difficulties arise when one tries to obtain an explicit
expression  of the action of the fractional Laplacian on a given function. In order to overcome these difficulties, we use the ideas by Caffarelli and Silvestre (see \cite{Caffarelli2007}) together with those of \cite{Braendle2013,Cabre2010,Capella2011}  to give an equivalent definition of the operator $(-\Delta)^s$ by means of an auxiliary problem that we introduce next.\newline
Given a domain $\Omega \subset \mathbb{R}^N$, we set the cylinder $\mathscr{C}_{\Omega}=\Omega\times(0,\infty)\subset\mathbb{R}_+^{N+1}$. We denote by $(x,y)$ those points that belong to $\mathscr{C}_{\Omega}$ and by $\partial_L\mathscr{C}_{\Omega}=\partial\Omega\times[0,\infty)$ the lateral boundary of the cylinder. Let us also denote by  $\Sigma_{\mathcal{D}}^*=\Sigma_{\mathcal{D}}\times[0,\infty)$ and $\Sigma_{\mathcal{N}}^*=\Sigma_{\mathcal{N}}\times[0,\infty)$ as well as $\Gamma^*=\Gamma\times[0,\infty)$. It is clear that, by construction,
\begin{equation*}
\Sigma_{\mathcal{D}}^*\cap\Sigma_{\mathcal{N}}^*=\emptyset\,, \quad \Sigma_{\mathcal{D}}^*\cup\Sigma_{\mathcal{N}}^*=\partial_L\mathscr{C}_{\Omega} \quad \mbox{and} \quad \Sigma_{\mathcal{D}}^*\cap\overline{\Sigma_{\mathcal{N}}^*}=\Gamma^*\,.
\end{equation*}
 Given a function $u\in H_{\Sigma_{\mathcal{D}}}^s(\Omega)$ we define its $s$-harmonic extension, denoted by $U (x,y)=E_{s}[u(x)]$, as the solution to the problem
\begin{equation*}
        \left\{
        \begin{array}{rlcl}
           -\text{div}(y^{1-2s}\nabla U (x,y))&\!\!\!\!=0  & & \mbox{ in } \mathscr{C}_{\Omega} , \\
          B(U(x,y) )&\!\!\!\!=0   & & \mbox{ on } \partial_L\mathscr{C}_{\Omega} , \\
          U(x,0)&\!\!\!\!=u(x)  & &  \mbox{ on } \Omega\times\{y=0\} ,
        \end{array}
        \right.
\end{equation*}
where
\begin{equation*}
B(U)=U\rchi_{\Sigma_{\mathcal{D}}^*}+\frac{\partial U}{\partial \nu}\rchi_{\Sigma_{\mathcal{N}}^* },
\end{equation*}
being $\nu$, with an abuse of notation\footnote{Let $\nu$ be the outward normal to $\partial\Omega$ and $\nu_{(x,y)}$ the outward normal to $\mathscr{C}_{\Omega}$ then, by construction, $\nu_{(x,y)}=(\nu,0)$, $y>0$.}, the outward normal to $\partial_L\mathscr{C}_{\Omega}$. The extension function belongs to the space
\begin{equation*}
\mathcal{X}_{\Sigma_{\mathcal{D}}}^s(\mathscr{C}_{\Omega}) : =\overline{\mathcal{C}_{0}^{\infty}
((\Omega\cup\Sigma_{\mathcal{N}})\times[0,\infty))}^{\|\cdot\|_{\mathcal{X}_{\Sigma_{\mathcal{D}}}^s(\mathscr{C}_{\Omega})}},
\end{equation*}
where we define
\begin{equation*}
\|\cdot\|_{\mathcal{X}_{\Sigma_{\mathcal{D}}}^s(\mathscr{C}_{\Omega})}^2:=\kappa_s\int_{\mathscr{C}_{\Omega}}\mkern-5mu y^{1-2s} |\nabla (\cdot)|^2dxdy,
\end{equation*}
with $\kappa_s=2^{2s-1}\frac{\Gamma(s)}{\Gamma(1-s)}$ being $\Gamma(s)$ the Gamma function.\newline
Note that $\mathcal{X}_{\Sigma_{\mathcal{D}}}^s(\mathscr{C}_{\Omega})$ is a Hilbert space equipped with the norm
$\|\cdot\|_{\mathcal{X}_{\Sigma_{\mathcal{D}}}^s(\mathscr{C}_{\Omega})}$ which is induced by the scalar product
\begin{equation*}
\langle U, V \rangle_{\mathcal{X}_{\Sigma_{\mathcal{D}}}^s(\mathscr{C}_{\Omega})}=\kappa_s
\int_{\mathscr{C}_{\Omega}}y^{1-2s} \langle\nabla U,\nabla V\rangle dxdy.
\end{equation*}
Moreover, the following inclusions are satisfied,
\begin{equation} \label{embedd}
\mathcal{X}_0^s(\mathscr{C}_{\Omega}) \subset \mathcal{X}_{\Sigma_{\mathcal{D}}}^s(\mathscr{C}_{\Omega}) \subsetneq \mathcal{X}^s(\mathscr{C}_{\Omega}),
\end{equation}
being  $\mathcal{X}_0^s(\mathscr{C}_{\Omega})$ the space of functions that belongs to $\mathcal{X}^s(\mathscr{C}_{\Omega})\equiv H^1(\mathscr{C}_{\Omega},y^{1-2s}dxdy)$ and vanish on the lateral boundary of $\mathscr{C}_{\Omega}$, denoted by $\partial_L\mathscr{C}_{\Omega}$.\newline
Following the well known result by Caffarelli and Silvestre (see \cite{Caffarelli2007}), $U$ is related to the fractional Laplacian of the original function through the formula
\begin{equation*}
\frac{\partial U}{\partial \nu^s}:= -\kappa_s\lim_{y\to 0^+} y^{1-2s}\frac{\partial U}{\partial y}=(-\Delta)^su(x).
\end{equation*} 
Using the above arguments we can reformulate the problem \eqref{problema1} in terms of the extension problem as follows:
\begin{equation}\label{extension_problem}
        \left\{
        \begin{array}{rlcl}
           -\text{div}(y^{1-2s}\nabla U)&\!\!\!\!=0  & & \mbox{ in } \mathscr{C}_{\Omega} , \\
         B(U)&\!\!\!\!=0   & & \mbox{ on } \partial_L\mathscr{C}_{\Omega} , \vspace{0.05cm}\\
         \displaystyle \frac{\partial U}{\partial \nu^s}&\!\!\!\!=f  & &  \mbox{ on } \Omega\times\{y=0 \},
        \end{array}
        \right.
        \tag{$P_{s}^*$}
\end{equation}
and we have that $u(x) = U(x,0)$.
\medskip

Next, we specify the meaning of solution to problem \eqref{extension_problem} and its relationship with the solutions to problem \eqref{problema1}.
\begin{definition}
An  energy solution to problem \eqref{extension_problem} is a function $U\in \mathcal{X}_{\Sigma_{\mathcal{D}}}^s(\mathscr{C}_{\Omega})$ such that
\begin{equation}\label{def:soldebil}
\kappa_s\int_{\mathscr{C}_{\Omega}} y^{1-2s} \nabla U\nabla\varphi dxdy=\int_{\Omega} f(x)\varphi(x,0)dx\quad \text{for all }  \varphi\in \mathcal{X}_{\Sigma_{\mathcal{D}}}^s(\mathscr{C}_{\Omega}).
\end{equation}
\end{definition}
If  $U\in \mathcal{X}_{\Sigma_{\mathcal{D}}}^s(\mathscr{C}_{\Omega})$  is the solution to problem \eqref{extension_problem}, we can associate the function $u (x) =Tr[U(x,y) ]=U(x,0)$, that  belongs to $H_{\Sigma_{\mathcal{D}}}^s(\Omega)$, and solves problem \eqref{problema1}.
Moreover, also the vice versa is true: given the solution $u \in H_{\Sigma_{\mathcal{D}}}^s(\Omega)$ to \eqref{problema1} its $s$-harmonic extension $U=E_s[u(x)]\in \mathcal{X}_{\Sigma_{\mathcal{D}}}^s(\mathscr{C}_{\Omega})$
is the solution to \eqref{extension_problem}.
Thus, both formulations are equivalent and the {\it Extension operator}
$$
E_s: H_{\Sigma_{\mathcal{D}}}^s(\Omega) \to \mathcal{X}_{\Sigma_{\mathcal{D}}}^s(\mathscr{C}_{\Omega}),
$$
allows us to switch between each other.\newline 
Moreover, according to \cite{Braendle2013, Caffarelli2007}, due to the choice of the constant $\kappa_s$, the extension operator $E_s$ is an isometry, i.e.,
\begin{equation}\label{norma2}
\|E_s[\varphi] (x,y) \|_{\mathcal{X}_{\Sigma_{\mathcal{D}}}^s(\mathscr{C}_{\Omega})}=
\|\varphi (x) \|_{H_{\Sigma_{\mathcal{D}}}^s(\Omega)}\quad \text{for all}\ \varphi\in H_{\Sigma_{\mathcal{D}}}^s(\Omega).
\end{equation}
Let us also recall the \textit{trace inequality} (cf. \cite{Braendle2013}) that is a useful tool to be exploited along this paper:
there exists $C=C(N,s,r,|\Omega|)$ such that for all $z\in\mathcal{X}_0^s(\mathscr{C}_{\Omega})$ we have
\begin{equation*}
C\left(\int_{\Omega}|z(x,0)|^rdx\right)^{\frac{2}{r}}\leq \int_{\mathscr{C}_{\Omega}}y^{1-2s}|\nabla z(x,y)|^2dxdy,
\end{equation*}
with  $1\leq r\leq 2^*_s,\ N>2s$, with $2^*_s= \frac{2N}{N-2s}$. 
Observe that, because of \eqref{norma2}, the trace inequality turns out to be, in fact, equivalent to the fractional Sobolev inequality:
\begin{equation}\label{sob}
  C\left(\int_{\Omega}|v|^rdx\right)^{\frac{2}{r}}\leq \int_{\Omega}|(-\Delta)^{\frac{s}2}v|^2dx
\qquad\text{for all } v\in H_{0}^s(\Omega),\ 1\leq r\leq 2^*_s ,\ N>2s.
\end{equation}
If $r=2_s^*$ the best constant in \eqref{sob}, namely the fractional Sobolev constant, denoted by $S(N,s)$, is independent of the domain $\Omega$ and its exact value is given by 
\begin{equation*}
S(N,s)=2^{2s}\pi^s\frac{\Gamma(\frac{N+2s}{2})}{\Gamma(\frac{N-2s}{2})}\left(\frac{\Gamma(\frac{N}{2})}{\Gamma(N)}\right)^{\frac{2s}{N}}.
\end{equation*}

When mixed boundary conditions are considered, the situation is quite similar since the Dirichlet condition is imposed on a set $\Sigma_{\mathcal{D}} \subset \partial \Omega$ such that $|\Sigma_{\mathcal{D}}|=\alpha>0$. Hence, thanks to \eqref{embedd}, there exists a positive constant $C_{\mathcal{D}}=C_{\mathcal{D}}(N,s,|\Sigma_{\mathcal{D}}|)$ such that
\begin{equation}\label{const}
0<C_{\mathcal{D}}\vcentcolon=\inf_{\substack{u\in H_{\Sigma_{\mathcal{D}}}^s(\Omega)\\ u\not\equiv 0}}\frac{\|u\|_{H_{\Sigma_{\mathcal{D}}}^s(\Omega)}^2}{
\|u\|_{L^{2_s^*}(\Omega)}^2}
<\inf_{\substack{u\in H_{0}^s(\Omega)\\ u\not\equiv 0}}\frac{\|u\|_{H_{0}^s(\Omega)}^2}{\|u\|_{L^{2_s^*}(\Omega)}^2}.
\end{equation}
Moreover, we have (cf. \cite[Proposition 3.6]{Colorado2019}),
\begin{equation*}
C_{\mathcal{D}}(N,s,|\Sigma_{\mathcal{D}}|)\leq2^{-\frac{2s}{N}}S(N,s).
\end{equation*}
\begin{remark}
Due to the spectral definition of the fractional operator, using H\"older's inequality, we have $C_{\mathcal{D}}\leq|\Omega|^{\frac{2s}{N}}\lambda_1^s(\alpha)$, where $\lambda_1(\alpha)$ denotes the first eigenvalue of the Laplace operator endowed with mixed boundary conditions on $\Sigma_{\mathcal{D}}=\Sigma_{\mathcal{D}}(\alpha)$ and $\Sigma_{\mathcal{N}}= \Sigma_{\mathcal{N}}(\alpha)$. Since $\lambda_1(\alpha)\to0$ as $\alpha\to0^+$, (cf. \cite[Lemma 4.3]{Colorado2003}), we have $C_{\mathcal{D}}\to0$ as $\alpha\to0^+$.
\end{remark}

Gathering together \eqref{const} and \eqref{norma2} it follows that, for all $\varphi \in \mathcal{X}_{\Sigma_{\mathcal{D}}}^s(\mathscr{C}_{\Omega})$,
\begin{equation*}
C_{\mathcal{D}}\left(\int_\Omega |\varphi(x,0)|^{2^*_s} dx\right)^{\frac{2}{2^*_s}}\leq\|\varphi(x,0)\|_{H_{\Sigma_{\mathcal{D}}}^s(\Omega)}^2=\|E_s[\varphi(x,0)]\|_{\mathcal{X}_{\Sigma_{\mathcal{D}}}^s(\mathscr{C}_{\Omega})}^2.
\end{equation*}
This Sobolev--type inequality provides a trace inequality adapted to the mixed boundary data framework.
\begin{lemma}\cite[Lemma 2.4]{Colorado2019}\label{lem:traceineq}
There exists a constant $C_{\mathcal{D}}=C_{\mathcal{D}}(N,s,|\Sigma_{\mathcal{D}}|)>0$ such that,
\begin{equation}\label{eq:traceineq}
C_{\mathcal{D}}\left(\int_\Omega|\varphi(x,0)|^{2^*_s}  )dx\right)^{\frac{2}{2^*_s}}\leq\int_{\mathscr{C}_{\Omega}} y^{1-2s} |\nabla \varphi|^2 dxdy\quad \text{for all }\varphi \in \mathcal{X}_{\Sigma_{\mathcal{D}}}^s(\mathscr{C}_{\Omega}).
\end{equation}
\end{lemma}
Along the proof of the Theorem \ref{teo:dav} we will make use of the following fractional Hardy inequality (cf. \cite[Theorem 3]{Filippas2013}): given $\frac12\leq s<1$, there exist a constant $C>0$ such that 
\begin{equation}\label{hardy_fract}
C\int_{\Omega}\left(\frac{f(x)}{d^{s}(x)}\right)^2dx\leq\|f\|_{H_0^s(\Omega)}^2\quad\text{for all }f\in H_0^s(\Omega),
\end{equation}
where $d(x)=\textrm{dist}(x,\partial\Omega)$.\newline
In order to establish the validity of \eqref{hardy_fract} we need to impose some geometrical or smoothness assumptions on the domain $\Omega$. From the geometrical point of view, if one assumes that $\Omega$ is such that, in the sense of distributions,
\begin{equation*}
-\Delta d(x)\geq 0 \quad\text{for } x\in\Omega,
\end{equation*}
then, inequality \eqref{hardy_fract} holds for the constant
$
C=C(s)=\frac{2^{2s}\Gamma^2(\frac{3+2s}{4})}{\Gamma^2(\frac{3-2s}{4})}.
$
The above condition is related to, but weaker than, the assumption of convexity of the domain $\Omega$. From the regularity point of view, if one considers a smooth domain $\Omega$, then inequality \eqref{hardy_fract} holds for a constant
$C\leq\frac{2^{2s}\Gamma^2(\frac{3+2s}{4})}{\Gamma^2(\frac{3-2s}{4})}$.
Finally, in terms of the $s$-harmonic extension the fractional Hardy inequality reads (cf. \cite[Theorem 1]{Filippas2013})
\begin{equation}\label{eq:hardy}
\int_\Omega \left(\frac{F(x,0)}{d^{s}(x)}\right)^2 dx\leq C\int_{\mathscr{C}_{\Omega}}y^{1-2s}|\nabla F(x,y)|^2 dxdy,
\end{equation}
for all $F\in \mathcal{X}_0^s(\mathscr{C}_{\Omega})$ and some constant $C>0$.

\section{Proof of main results}
In this section we prove Theorem \ref{lem:davila2} and, as a consequence, Theorem \ref{teo:dav}. Following the approach of \cite{Davila2001}, we start by proving the following weighted Sobolev--type inequality.

\begin{lemma}\label{lem:wsob}
Let $u\in H_{\Sigma_{\mathcal{D}}}^s(\Omega)$ be the solution to \eqref{problema1} with $f\in L^\infty(\Omega)$, $f\gneq 0$ and denote by $U\in  \mathcal{X}_{\Sigma_{\mathcal{D}}}^s(\mathscr{C}_{\Omega})$ its $s$-harmonic extension. Then, there exists a constant $C>0$ such that, for every $\varphi\in\mathcal{X}_{\Sigma_{\mathcal{D}}}^s(\mathscr{C}_{\Omega})\cap L^{\infty}(\mathscr{C}_{\Omega})$,
 \begin{equation}\label{eq:claim}
\left(\int_\Omega u^r(x) |\varphi(x,0)|^qdx \right)^\frac2q\leq C\left(\int_{\mathscr{C}_{\Omega}}  y^{1-2s} U^2 |\nabla \varphi|^2 dxdy + \int_\Omega
u^2(x) \varphi^2(x,0)dx\right),
\end{equation}
where $0\leq r\leq 2^*_s$ and $\frac{q}{2}=1+\frac{rs}{N}$ and the constant $C>0$ depends on $N$, $s$, $\Omega$, $\|u\|_{L^\infty(\Omega)}$, $\|f\|_{L^\infty(\Omega)}$ and $1/(\int_{\Omega}f(z)d(z)dz)$.
\end{lemma}

\begin{proof}
We divide the proof into three steps according to the cases $r=0$, $r=2^*_s$ and the interpolation case $r\in (0,2^*_s)$.

\medskip
\noindent {\underline{\bf Step 1}: Case $r=0$.}\newline
We start by proving that, for all $\varphi \in \mathcal{X}_{\Sigma_{\mathcal{D}}}^s(\mathscr{C}_{\Omega})\cap L^{\infty}(\mathscr{C}_{\Omega})$,
\begin{equation*}
\int_\Omega \varphi^2(x,0)dx\leq C \left( \int_{\mathscr{C}_{\Omega}} y^{1-2s} U^2 |\nabla \varphi|^2 dxdy + \int_\Omega u^2(x) \varphi^2(x,0)dx\right).
\end{equation*}
Let $\phi_1$ be the first eigenfunction of $(-\Delta)^s$ under homogeneous Dirichlet boundary condition, 
\begin{equation*}
        \left\{
        \begin{tabular}{rl}
        $(-\Delta)^s\phi_{1}=\lambda_1^s \phi_{1}$ & in $\Omega$, \\
        $\phi_{1}=0\mkern+28mu$ & on $\partial\Omega$.
        \end{tabular}
        \right.
\end{equation*}
As in \eqref{def:soldebil},  in terms of the $s$-harmonic extension of $\phi_1$, denoted by $\Phi_1=E_s[\phi_1]$, we have 
\begin{equation}\label{eq:fi1}
\kappa_s\int_{\mathscr{C}_{\Omega}} y^{1-2s} \nabla\Phi_1 \nabla \Psi dxdy=\lambda_1^s\int_\Omega \phi_{1}(x) \Psi(x,0)dx \quad \text{for all } \Psi \in\mathcal{X}_0^s(\mathscr{C}_{\Omega}).
\end{equation}
Let us remark that $\Phi_1\in \mathcal{C}^{\gamma}(\overline{\mathscr{C}}_{\Omega})\cap L^{\infty}(\mathscr{C}_{\Omega})$ for some $\gamma\in(0,1)$ (cf. \cite[Theorem 4.7]{Braendle2013} and \cite[Corollary 4.8]{Braendle2013}). Because of the spectral definition of the fractional operator, the function $\phi_1$ is also the first eigenfunction of the classical Laplace operator $(-\Delta)$ under homogeneous Dirichlet boundary condition, hence, there exists a constant $c_1>0$ (depending only on $\Omega$) such that
\begin{equation}\label{eq:rel:fi1:dist}
\phi_1(x) \geq c_1 d(x) \quad  \mbox{for } x\in\Omega.
\end{equation}
Using \eqref{eq:rel:fi1:dist} and \eqref{eq:hardy} with $F=\varphi\Phi_1^s$ (note that $\varphi\Phi_1^s\in\mathcal{X}_{\Sigma_{\mathcal{D}}}^s(\mathscr{C}_{\Omega})$), we get
\begin{equation}\label{ineq:dist}
\int_\Omega \varphi^2(x,0) \, dx \leq c_1\int_\Omega \frac{\varphi^2(x,0) \phi_1^{2s} (x) }{d^{2s}(x)}\, dx\leq  C\int_{\mathscr{C}_{\Omega}}y^{1-2s}|\nabla (\varphi \Phi_1^s)|^2 dxdy.
\end{equation}
Next, we observe that 
\begin{equation}\label{est}
\begin{split}
|\nabla (\varphi \Phi_1^s)|^2=&\Phi_1^{2s}|\nabla \varphi|^2+2s\varphi\Phi_1^{2s-1}\nabla\varphi\nabla\Phi_1+s^2\varphi^2\Phi_1^{2s-2}|\nabla \Phi_1|^2\\
=&\Phi_1^{2s}|\nabla \varphi|^2+s\nabla\Phi_1\nabla(\varphi^2\Phi_1^{2s-1})+s(1-s)\varphi^2\Phi_1^{2s-2}|\nabla\Phi_1|^2.
\end{split}
\end{equation}
On the other hand, since $\frac12<s<1$, using the Cauchy--Schwarz and $\varepsilon$--Young inequalities, we get
\begin{align*}
\nabla\Phi_1\nabla(\varphi^2\Phi_1^{2s-1})=&\,2\varphi\Phi_1^{2s-1}\nabla\Phi_1\nabla\varphi+(2s-1)\varphi^2\Phi_1^{2s-2}|\nabla\Phi_1|^2\\
\geq&-\frac{1}{\varepsilon}\Phi_1^{2s}|\nabla\varphi|^2+(2s-1-\varepsilon)\varphi^2\Phi_1^{2s-2}|\nabla\Phi_1|^2.
\end{align*}
for some $\varepsilon>0$ such that $2s-1-\varepsilon>0$. Then,
\begin{equation*}
\varphi^2\Phi_1^{2s-2}|\nabla\Phi_1|^2\leq \frac{1}{2s-1-\varepsilon}\left(\frac{1}{\varepsilon}\Phi_1^{2s}|\nabla\varphi|^2+\nabla\Phi_1\nabla(\varphi^2\Phi_1^{2s-1})\right).
\end{equation*}
As a consequence, from \eqref{est}, we find
\begin{equation*}
\begin{split}
|\nabla (\varphi \Phi_1^s)|^2\leq \left(1+\frac{s(1-s)}{\varepsilon(2s-1-\varepsilon)}\right)\Phi_1^{2s}|\nabla \varphi|^2+\left(s+\frac{s(1-s)}{2s-1-\varepsilon}\right)\nabla\Phi_1\nabla(\varphi^2\Phi_1^{2s-1}).
\end{split}
\end{equation*}
Therefore, because of \eqref{ineq:dist} and \eqref{eq:fi1},
\begin{equation*}
\begin{split}
\int_\Omega \varphi^2(x,0)\, dx &\leq C\int_{\mathscr{C}_{\Omega}}y^{1-2s}|\nabla (\varphi \Phi_1^s)|^2 dxdy\\
&\leq C_1\int_{\mathscr{C}_{\Omega}}y^{1-2s}\Phi_1^{2s}|\nabla \varphi|^2dxdy+C_2\int_{\mathscr{C}_{\Omega}}y^{1-2s}\nabla\Phi_1\nabla(\varphi^2\Phi_1^{2s-1})dxdy\\
&= C_1\int_{\mathscr{C}_{\Omega}}y^{1-2s}\Phi_1^{2s}|\nabla \varphi|^2dxdy+\frac{C_2\lambda_1^s}{\kappa_s}\int_\Omega \phi_{1}^{2s}(x) \varphi^2(x,0)dx.
\end{split}
\end{equation*}
Finally, given $u$ the solution to \eqref{problema1}, because of the H\"older regularity of solutions to fractional elliptic problems with mixed boundary data (cf. \cite[Theorem 1.1]{Carmona2020}), we have $u\in \mathcal{C}^{\gamma}(\overline{\Omega})$ for some $\gamma\in(0,\frac12)$. Moreover, since $\Omega$ is a smooth bounded domain, we also have $\phi_1\in\mathcal{C}_0^{\infty}(\overline{\Omega})$ and, thus, $\phi_1^s\in\mathcal{C}^s(\overline{\Omega})$. As consequence, since $\frac12<s<1$, there exists a constant $C>0$ such that $\phi_{1}^{s} \leq C u $ and, hence, $\Phi_1^s\leq C U$. Then, we conclude
\begin{equation*}
\int_\Omega \varphi^2(x,0)\, dx\leq C\left(\int_{\mathscr{C}_{\Omega}}y^{1-2s}U^2|\nabla \varphi|^2dxdy+\int_\Omega u^2(x) \varphi^2(x,0)dx\right),
\end{equation*}
for some constant $C>0$.

\medskip 
\noindent {\underline{\bf Step 2}: Case $r=2_s^*$.}\newline
We continue by proving that, for all $\varphi \in \mathcal{X}_{\Sigma_{\mathcal{D}}}^s(\mathscr{C}_{\Omega})\cap L^{\infty}(\mathscr{C}_{\Omega})$,
\begin{equation*}
\left(\int_\Omega\left(u(x)|\varphi(x,0)|\right)^{2^*_s}dx\right)^{\frac{2 }{2^*_s} }\leq C\left( \int_{\mathscr{C}_{\Omega}} y^{1-2s} U^2 |\nabla \varphi|^2 dxdy+\int_\Omega u^2(x) \varphi^2(x,0)dx\right).
\end{equation*}
Since by hypothesis $f\in L^{\infty}(\Omega)$, repeating step by step the Moser--type proof done for fractional elliptic problems with Dirichlet boundary data (cf. \cite[Theorem 4.7]{Braendle2013}), we get that $U=E_s[u]\in L^{\infty}(\mathscr{C}_{\Omega})$, being $u$ the solution to \eqref{problema1}. Thus, $U\varphi^2 \in\mathcal{X}_{\Sigma_{\mathcal{D}}}^s(\mathscr{C}_{\Omega})$ and, because of \eqref{eq:traceineq} and \eqref{def:soldebil}, we obtain (the constants may vary line to line)
\begin{align*}
\left(\int_\Omega (u(x)|\varphi(x,0)|)^{2^*_s}dx\right)^{\frac{2}{2^*_s}}&\leq  C\int_{\mathscr{C}_{\Omega}}y^{1-2s}|\nabla (U \varphi)|^2 dxdy \\
&=C\int_{\mathscr{C}_{\Omega}} y^{1-2s}U^2|\nabla \varphi|^2dxdy +
C\int_{\mathscr{C}_{\Omega}} y^{1-2s}\nabla U \nabla (U\varphi^2)dxdy\\
&= C \int_{\mathscr{C}_{\Omega}} y^{1-2s}U^2|\nabla \varphi|^2dxdy + C
\int_{\Omega} f(x) u(x) \varphi^2(x,0)dx\\
&\leq C  \int_{\mathscr{C}_{\Omega}} y^{1-2s}U^2|\nabla \varphi|^2dxdy +
C\|fu\|_{L^\infty (\Omega)}\int_{\Omega} \varphi^2(x,0)dx\\
&\leq  C \left(\int_{\mathscr{C}_{\Omega}} y^{1-2s}U^2|\nabla \varphi|^2dxdy + 
\int_{\Omega} u^2(x) \varphi^2(x,0)dx\right),
\end{align*} 
where we have used that $u\in L^{\infty}(\Omega)$  (cf. \cite[Theorem 3.7]{Carmona2020}) and Step 1 in the last inequality.

\medskip 
\noindent {\underline{\bf Step 3}: Case $r\in(0,2_s^*)$. }\newline
Finally, we prove inequality \eqref{eq:claim}. By H\"older's inequality, 
Step 1 and Step 2 we conclude
\begin{align*}
 \int_\Omega u^r(x) |\varphi (x,0)|^q dx=& \int_\Omega u^r(x) |\varphi|^{r}(x,0)|\varphi|^{2+\frac{2rs}N-r}(x,0)dx\\
 \leq & \left(\int_\Omega \varphi^{2}(x,0)dx\right)^{1+\frac{rs}N-\frac r2} \left(\int_\Omega u^{2^*_s}(x) |\varphi (x,0) |^{2^*_s}dx\right)^{\frac{r}{2^*_s}}\\
 \leq & C\left( \int_{\mathscr{C}_{\Omega}} y^{1-2s} U^2 |\nabla \varphi|^2 dxdy + \int_\Omega u^2(x) \varphi^2(x,0)dx\right)^{\frac q2},
\end{align*}
since $\frac q2=1+\frac{rs}N$.
\end{proof}

\begin{proof}[Proof of Theorem \ref{lem:davila2}]
First, we observe that it is enough to prove the result in the case $g\geq 0$. The general case is deduced applying this argument to
the positive and negative parts of $g$ respectively.

Let $u,v\in H_{\Sigma_{\mathcal{D}}}^s(\Omega)$ be the solutions to \eqref{problema1} and \eqref{prob_v_com_1} respectively. Then, $U=E_s[u],V=E_s[v]\in\mathcal{X}_{\Sigma_{\mathcal{D}}}^s(\mathscr{C}_{\Omega})$ are the respective solutions to the extension problems
\begin{equation*}
        \left\{
        \begin{array}{rll}
           -\text{div}(y^{1-2s}\nabla U)&\!\!\!\!=0  &  \mbox{ in } \mathscr{C}_{\Omega} , \\
         B(U)&\!\!\!\!=0   &  \mbox{ on } \partial_L\mathscr{C}_{\Omega} , \\
         U(x,0)&\!\!\!\!=u(x)  &   \mbox{ on } \Omega\times\{y=0 \},\vspace{0.05cm}\\
        \displaystyle \frac{\partial U}{\partial \nu^s}&\!\!\!\!=f  &  \mbox{ on } \Omega\times\{y=0 \},
        \end{array}
        \right.
\end{equation*}
and 
\begin{equation*}
        \left\{
        \begin{array}{rll}
           -\text{div}(y^{1-2s}\nabla V)&\!\!\!\!=0  & \mbox{ in } \mathscr{C}_{\Omega} , \\
         B(V)&\!\!\!\!=0   & \mbox{ on } \partial_L\mathscr{C}_{\Omega} , \\
         V(x,0)&\!\!\!\!=v(x)  &  \mbox{ on } \Omega\times\{y=0 \},\vspace{0.05cm}\\
        \displaystyle \frac{\partial V}{\partial \nu^s}&\!\!\!\!=g  &  \mbox{ on } \Omega\times\{y=0 \}.
        \end{array}
        \right.
\end{equation*}
Taking in mind \eqref{def:soldebil}, for every test function $\varphi \in\mathcal{X}_{\Sigma_{\mathcal{D}}}^s(\mathscr{C}_{\Omega})$, we have
\begin{equation}\label{eq1}
\kappa_s\int_{\mathscr{C}_{\Omega}} y^{1-2s}\nabla U\nabla \varphi dxdy=\int_{\Omega}\varphi(x,0) f(x)dx,
\end{equation}
and 
\begin{equation}\label{eq2}
\kappa_s\int_{\mathscr{C}_{\Omega}} y^{1-2s}\nabla V\nabla \varphi dxdy=\int_{\Omega}\varphi(x,0) g(x)dx.
\end{equation}
As we commented before, since $f\in L^{\infty}(\Omega)$ and $g\in L^p(\Omega)$, with $p>N/s$, repeating step by step the proof of \cite[Theorem 4.7]{Braendle2013} we get that $U,V\in L^{\infty}(\mathscr{C}_{\Omega})$. Let us also stress that $u,v\in L^{\infty}(\Omega)$ by \cite[Theorem 3.7]{Carmona2020}. Moreover, since $g\geq0$, by comparison with the respective Dirichlet problem, we can assume that $v\geq0$ and $V=E_s[v]\geq0$ (cf. \cite[Lemma 2.3]{Capella2011}).\newline
Then, for $\varepsilon>0$ and $k\geq0$, we define 
\begin{equation*}
\varphi_\varepsilon=\left(\frac{V}{U+\varepsilon}-k\right)_+\in \mathcal{X}_{\Sigma_{\mathcal{D}}}^s(\mathscr{C}_{\Omega})\cap L^{\infty}(\mathscr{C}_{\Omega}),
\end{equation*}
where $(\cdot)_+=\max\{0,\cdot\}$. Since $\varphi_{\varepsilon}$ is bounded, both $V\varphi_\varepsilon$ and $U \varphi_\varepsilon$ belong to $\mathcal{X}_{\Sigma_{\mathcal{D}}}^s(\mathscr{C}_{\Omega})$. Then, using $V\varphi_\varepsilon$ and $U \varphi_\varepsilon$ as a test function in \eqref{eq1} and \eqref{eq2} respectively, we get
\begin{equation*}
\kappa_s\int_{\mathscr{C}_{\Omega}} y^{1-2s} \varphi_\varepsilon \nabla U\nabla V dxdy+ \kappa_s\int_{\mathscr{C}_{\Omega}} y^{1-2s}V \nabla U\nabla \varphi_\varepsilon dxdy =  \int_{\Omega} \varphi_\varepsilon(x,0) v(x) f(x)dx,
\end{equation*}
and 
\begin{equation*}
\kappa_s\int_{\mathscr{C}_{\Omega}} y^{1-2s} \varphi_\varepsilon \nabla V\nabla U dxdy+ \kappa_s\int_{\mathscr{C}_{\Omega}} y^{1-2s} U \nabla V \nabla \varphi_\varepsilon dxdy =  \int_{\Omega} \varphi_\varepsilon(x,0)u(x) g(x)dx.
\end{equation*}
Hence, subtracting the above equalities,
\begin{equation}\label{mix}
 \kappa_s\int_{\mathscr{C}_{\Omega}}  y^{1-2s}(U \nabla V-V\nabla U)\nabla \varphi_\varepsilon dxdy= \int_{\Omega} \varphi_\varepsilon(x,0) (u(x)g(x)-v(x)f(x))dx\,.
\end{equation}
We observe now that
\begin{equation*}
(U+\varepsilon)^2|\nabla \varphi_\varepsilon|^2=(U \nabla V-V\nabla U)\nabla \varphi_\varepsilon 
+\varepsilon \nabla V \nabla\varphi_\varepsilon, 
\end{equation*}
and consequently, by \eqref{mix},
\begin{equation}\label{pre}
\begin{split}
\kappa_s\int_{\mathscr{C}_{\Omega}}y^{1-2s} (U+\varepsilon)^2|\nabla\varphi_\varepsilon|^2 dxdy=&\,\varepsilon  \kappa_s\int_{\mathscr{C}_{\Omega}}  y^{1-2s}
\nabla V\nabla \varphi_\varepsilon dxdy\\
&+\int_{\Omega} \varphi_\varepsilon(x,0) (u(x)g(x)-v(x)f(x))dx\\
=&\int_{\Omega} \varphi_\varepsilon(x,0)\big((u(x)+\varepsilon)g(x)-v(x)f(x)\big)dx\\
\leq&\int_{\Omega} \varphi_\varepsilon(x,0)
(u(x)+\varepsilon)g(x)dx.
\end{split}
\end{equation}
Finally, because of Lemma \ref{lem:wsob}, inequality \eqref{pre} reads as
\begin{equation*}
\left(\int_\Omega u^r(x)
|\varphi_\varepsilon(x,0)|^qdx\right)^\frac2q\leq C\left(
\int_{\Omega} \varphi_\varepsilon(x,0) (u(x)+\varepsilon)g(x)dx +
\int_{\Omega} u^2(x)\varphi_\varepsilon^2(x,0)dx \right),
\end{equation*}
with $\frac q2=1+\frac{rs}N$. On the other hand, we observe that, for $\varepsilon\to 0$, 
\begin{equation*}
(u(x)+\varepsilon)\varphi_\varepsilon(x,0)=(v-k(u+\varepsilon))_+\nearrow(v-ku)_+
\end{equation*}
and 
\begin{equation*}
\varphi_\varepsilon(x,0)\nearrow\left(\frac vu -k \right)_+.
\end{equation*}
Thus, denoting $w=\frac vu$, by the monotone convergence theorem we obtain
\begin{equation}\label{eq3}
\left(\int_{\Omega} u^r (w-k)_+^qdx\right)^\frac2q\leq C\left( \int_{\Omega} u(w-k)_+gdx + \int_{\Omega} u^2 (w-k)_+^2dx \right).
\end{equation}
Once we get inequality \eqref{eq3}, the rest of the proof follows by means of an iterative Stampacchia--type method. We include the argument for the reader's convenience.\newline
Let us set $r=\frac p{p-1}\in(1,2_s^*)$. Thus, using H\"older's inequality, from \eqref{eq3} we obtain that
\begin{equation}\label{eq:inter}
\begin{split}
\left(\int_{\Omega} u^r (w-k)_+^q dx\right)^\frac2q \leq&\ C\left( \int_{\Omega} u^{1-\frac rq}u^{\frac rq}(w-k)_+gdx
+\int_{\Omega} u^{2-\frac{2r}q}u^{\frac{2r}q} (w-k)_+^2dx \right)\\
\leq &\ C\|g\|_{L^p (\Omega)}\left(\int_\Omega u^r (w-k)_+^q dx\right)^{\frac1q}\left(\int_{\{w>k\}} u
^rdx\right)^{1-\frac1p-\frac1q}\\
&+ C \left(\int_{\Omega}u^r (w-k)_+^q dx\right)^{\frac2q}\left(\int_{\{w>k\}}u^{\frac{2(q-r)}{q-2}}dx\right)^{1-\frac2q}.
\end{split}
\end{equation}
Observe that $\frac q2=1+\frac{rs}N>1$ and $\frac{2(q-r)}{q-2}>0$ since, as $r<2_s^*$, we have $\frac qr = \frac2r+\frac{2s}N>1$. Moreover,
\begin{equation*}
\int_{\{w>k\}} u^{\frac{2(q-r)}{q-2}}dx\leq \int_{\{w>k\}} \left(\frac
vk\right)^{\frac{2(q-r)}{q-2}}dx\leq C \left(\frac
{\|v\|_{L^\infty(\Omega)}}{k}\right)^{\frac{2(q-r)}{q-2}}.
\end{equation*}
Let us define 
\begin{equation}\label{eq4}
k_0\vcentcolon=(2C)^{\frac {q}{2(q-r)}}\|v\|_{L^\infty(\Omega)},
\end{equation}
so that, for $k\geq k_0$, from \eqref{eq:inter} we get
\begin{equation*}
\left(\int_{\Omega} u^r ((w-k)_+)^qdx\right)^\frac{1}{q} \leq C\|g\|_{L^p}\left(\int_{\{w>k\}} u ^rdx\right)^{1-\frac1p-\frac1q}.
\end{equation*}
Using H\"older's inequality once more, we deduce
\begin{equation}\label{eq5}
\begin{split}
\int_{\Omega} u^r(w-k)_+dx&\leq \left(\int_{\Omega} u^r((w-k)_+)^qdx\right)^\frac{1}{q} \left(\int_{\{w>k\}} u^rdx\right)^{1-\frac{1}{q}}\\
&\leq C\|g\|_{L^p(\Omega)} \left(\int_{\{w>k\}} u^rdx\right)^{2-\frac1p-\frac2q},
\end{split}
\end{equation}
where $2-\frac1p-\frac2q=1+\frac1r-\frac2q>1$ since $p>\frac Ns$. Next, we set the function 
\begin{equation*}
a(k)\vcentcolon=\int_{\Omega} u^r(w-k)_+dx=\int_k^\infty\int_{\{w>t\}}u^rdxdt,
\end{equation*}
which satisfies 
\begin{equation*}
a'(k)=-\int_{\{w>k\}} u^rdx.
\end{equation*}
Then, denoting by $\gamma=2-\frac1p-\frac2q>1$, from \eqref{eq5} we find
\begin{equation*}
a(k)\leq C\|g\|_{L^p(\Omega)} (-a'(k))^\gamma\quad\text{for all } k\geq k_0.
\end{equation*}
Therefore, given  $k>k_0$ and integrating in the interval $[k_0,k]$, we get
\begin{equation*}
a^{1-\frac1\gamma}(k)\leq a^{1-\frac1\gamma}(k_0)-\frac{k-k_0}{C^{\frac1{\gamma}}
\|g\|_{L^p(\Omega)}^{\frac1{\gamma}}}.
\end{equation*}
Since $a(k)$ is nonnegative and nondecreasing and $\gamma>1$ the above inequality implies that $a(k)=0$ for some 
$k\leq C\|g\|_{L^p(\Omega)}^{\frac1\gamma}a^{1-\frac1\gamma}(k_0)+k_0$ and, hence,
\begin{equation}\label{eq6}
w \leq C\|g\|_{L^p(\Omega)}^{\frac1\gamma} a^{1-\frac1\gamma}(k_0)+k_0.
\end{equation}
In addition, from \eqref{eq5}, we have
\begin{equation*}
a(k_0)\vcentcolon=\int_{\Omega} u^r(w-k_0)_+dx\leq C\|g\|_{L^p(\Omega)} \left(\int_{\{w>k_0\}} u^rdx\right)^{2-\frac1p-\frac2q}\leq C\|g\|_{L^p(\Omega)}.
\end{equation*}
Then, using \eqref{eq6} and \eqref{eq4}, we conclude
\begin{equation*}
\frac{v}{u}=w\leq C(\|g\|_{L^p(\Omega)} +\|v\|_{L^{\infty}(\Omega)} )\leq C\|g\|_{L^p(\Omega)}, 
\end{equation*}
since, by \cite[Theorem 3.7]{Carmona2020}, we have $\|v\|_{L^\infty(\Omega)}\leq C(N,s,|\Sigma_{\mathcal{D}}|)|\Omega|^{\frac{2s}{N}-\frac{1}{p}}\|g\|_{L^p(\Omega)}$.
\end{proof}
Using Theorem \ref{lem:davila2} we can now prove Theorem \ref{teo:dav}.

\begin{proof}[Proof of Theorem \ref{teo:dav}]
First, we observe the following: Since $f\in C^\infty_0(\Omega)$ we have that $(-\Delta)^{1-s}f$ is bounded and we can choose $c_f>0$ such that  
\begin{equation*}
(-\Delta)^{1-s}f+c_f\geq0 \quad\text{in }\Omega,
\end{equation*}
and $w$ such that 
\begin{equation}\label{eq:wcero}
\left\{
        \begin{tabular}{rcl}
        $(-\Delta)w=c_f$ & &\mbox{in } $\Omega$, \\
        $w=0\mkern+7mu$  & &on $\partial\Omega$,
        \end{tabular}
\right.
\end{equation}
At one hand, fix $x_0 \in\Omega$ and let $\rho<\frac{1}{4}\textrm{dist}(x_0,\partial\Omega)$.  Let us recall that, for $w$ satisfiying \eqref{eq:wcero}, we have (cf. \cite[p.71 Problem 4.5]{Gilbarg1983})
\begin{equation}\label{eq:wuno}
w(x)=\frac{1}{|B_{2\rho}(x)|}\int_{B_{2\rho}(x)}w(z)dz+\frac {2c_f }{N(N+2)}\rho^2\quad\text{for }x\in B_{\rho}(x_0).
\end{equation}
Moreover, since $f\in\mathcal{C}_0^{\infty}(\Omega)$, by interior regularity (cf. \cite[Lemma 2.9]{Capella2011}, \cite[Lemma 4.4]{Cabre2014}) we have $u\in\mathcal{C}^{\infty}(B_{2\rho}(x))$ for any $x\in B_{\rho}(x_0)$. Thus, $(-\Delta)u$ is well defined in $B_{2\rho}(x)$ and 
\begin{equation*}
(-\Delta) (u+w)=(-\Delta)^{1-s}f+(-\Delta)w\geq0\quad\text{in }B_{2\rho}(x),
\end{equation*}
for any $x\in B_{\rho}(x_0)$. Then, as $u+w$ is superharmonic,
\begin{equation*}
u(x)+w(x)\geq\frac{1}{|B_{2\rho}(x)|}\int_{B_{2\rho}(x)}(u(z)+w(z))dz\quad \text{for }x\in B_{\rho}(x_0).
\end{equation*}
and, hence, by \eqref{eq:wuno}, 
\begin{equation}\label{eq:wdosb}
u(x)+\frac {2c_f }{N(N+2)}\rho^2\geq\frac{1}{|B_{2\rho}(x)|}\int_{B_{2\rho}(x)}u(z)dz\quad \text{for }x\in B_{\rho}(x_0).
\end{equation}

On the other hand, since $u\in\mathcal{C}^{\gamma}(\overline{\Omega})$ for some $\gamma\in(0,\frac12)$ (cf. \cite[Theorem 1.2]{Carmona2020}) and $\frac12<s<1$, there exists a constant $c_1>0$ such that $u(x)\geq c_1d^s(x)$ in the whole $\Omega$.
Then, we can choose $\rho$ small enough in order to have $c_1\rho^s \geq \frac {4 c_f}{N(N+2)}\rho^2$ and hence
\begin{equation*}
\frac{1}{|B_{2\rho}(x)|}\int_{B_{2\rho}(x)} u(z)dz \geq \frac{1}{|B_{2\rho}(x)|}\int_{B_{2\rho}(x)} c_1 d^s(z)dz \geq c_1\rho^s \geq \frac {4 c_f}{N(N+2)}\rho^2,
\end{equation*}
since $\textrm{dist}(\partial B_{2\rho}(x),\partial\Omega)\geq\rho$ for any $x\in B_{\rho}(x_0)$. Then, by \eqref{eq:wdosb}, we get
\begin{equation}\label{eq:wtres}
u(x)\geq \frac{1}{|B_{2\rho}(x)|}\int_{B_{2\rho}(x)}u(z)dz- \frac{2c_f}{N(N+2)} \rho^2\geq\frac{1}{2|B_{2\rho}(x)|}\int_{B_{2\rho}(x)}u(z)dz .
\end{equation}
Now we proceed as in \cite[Theorem 1]{Davila2001} which in turn is based on \cite[Lemma 3.2]{Brezis1998}.\newline  
Fixed $x_0\in \Omega$ and $\rho<\min\left\{\frac14d(x_0),\left(\frac{c_1N(N+2)}{4c_f}\right)^\frac1{2-s}\right\}$, we consider $f_0\in C_0^\infty(B_\rho(x_0))$ with $0\leq f_0\leq 1$, $f\not\equiv0$. Let $u_0$ be the solution to
\begin{equation*}
\left\{
        \begin{tabular}{rcl}
        $(-\Delta)^su_0=f_0$ & &\mbox{in } $\Omega$, \\
        $B(u_0)=0\mkern+7mu$  & &on $\partial\Omega$.
        \end{tabular}
\right.
\end{equation*}
Using Theorem \ref{lem:davila2}, there exists a positive constant $C$, depending only on $\Omega$, $\|u_0\|_{L^{\infty}(\Omega)}$ and $\|v\|_{L^{\infty}(\Omega)}$ such that $u_0\geq C_0 v$ in $\Omega$. Next, for $x\in B_{2\rho}(x_0)$ we have $\textrm{supp}(f_0)\subset B_{2\rho}(y)\subset \Omega$. Then, because of \eqref{eq:wtres}, we have for every $x\in\overline B_{\rho}(x_0)$
\begin{align*}
u(x)\geq&\frac1{2|B_{2\rho} (x)|}\int_{B_{2\rho}(x)} u(z)dz\geq\frac1{2|\Omega|} \int_{\Omega} u f_0dz\\
\geq& c'\int_{\Omega} uf_0dz=c'\int_{\Omega} f u_0dz\geq c'' \int_{\Omega} f vdz\\
\geq&  \lambda u_0(x),
\end{align*}
where 
\begin{equation*}
\lambda\vcentcolon=\frac{c''}{\|u_0\|_{L^{\infty}(\Omega)}}\left(\int_\Omega f vdx\right).
\end{equation*} 
Then $u-\lambda u_0\geq 0$ in $\overline B_\rho(x_0)$ and, in particular, $u-\lambda u_0\geq 0$ in $\partial B_\rho(x_0)$. Therefore, 
\begin{equation*}
\left\{
        \begin{tabular}{rcl}
        $(-\Delta)^s(u-\lambda u_0)=f\geq 0$ & &\mbox{in } $\Omega\setminus \overline B_\rho(x_0)$, \\
        $u-\lambda u_0\geq 0\mkern+35mu$   & &on $\partial B_\rho(x_0)$,\\
        $B(u-\lambda u_0)=0\mkern+35mu$  & &on $\partial\Omega$.
        \end{tabular}
\right.
\end{equation*}
Thus, by comparison, $u-\lambda u_0\geq 0$ in $\Omega\setminus \overline B_\rho(x_0)$ and, hence,
\begin{equation*}
u(x)\geq c''' \left(\int_\Omega f vdz\right) u_0(x) \geq c\left(\int_\Omega f vdz\right) v(x)\quad\text{for all }\forall x\in \Omega,
\end{equation*}
which gives us the desired conclusion.
\end{proof}

\end{document}